\documentclass[a4paper,12pt]{article}

\usepackage[T2A]{fontenc}
\usepackage[utf8]{inputenc}
\usepackage[russian]{babel}
\usepackage{amsmath, amssymb, amsthm}
\usepackage{enumerate}
\usepackage{indentfirst}
\usepackage[colorlinks=true,linkcolor=blue,citecolor=blue]{hyperref}

\usepackage{tikz}
\usepackage{tikz-3dplot}
\usetikzlibrary{fadings}
\tikzfading[name=fade out,
            inner color=transparent!0,
            outer color=transparent!100]
\usetikzlibrary{calc,arrows,decorations.pathreplacing,fadings,3d,positioning}

\title{Симметрии двумерной цепной дроби}
     \date{}
\title{Симметрии двумерной цепной дроби.
       \thanks{Работа выполнена при поддержке РФФИ (грант \textnumero 18-01-00886), а также гранта Фонда развития теоретической физики и математики «БАЗИС»}}
\author{О.\,Н.\,Герман, И.\,А.\,Тлюстангелов}
\date{}

\theoremstyle{definition}
\newtheorem{definition}{Определение}

\newtheorem*{notation*}{Обозначение}

\theoremstyle{remark}
\newtheorem{remark}{Замечание}
\newtheorem*{remark*}{Замечание}

\theoremstyle{plain}
\newtheorem{theorem}{Теорема}

\newtheorem{lemma}{Лемма}

\newtheorem{proposition}{Предложение}
\newtheorem{corollary}{Следствие}
\newtheorem*{statement*}{Утверждение}
\newtheorem*{corollary*}{Следствие}

\newtheorem{proof_m*}{Доказательство теоремы 1}

\DeclareMathOperator{\conv}{conv}
\DeclareMathOperator{\id}{id}

\renewcommand{\vec}[1]{\mathbf{#1}}
\renewcommand{\geq}{\geqslant}

\newcommand{\e}{\varepsilon}

\newcommand{\R}{\mathbb{R}}
\newcommand{\Z}{\mathbb{Z}}
\newcommand{\Q}{\mathbb{Q}}

\newcommand{\Comp}{\mathbb{C}}

\newcommand{\cC}{\mathcal{C}}

\newcommand{\cK}{\mathcal{K}}

\newcommand{\gO}{\mathfrak{O}}
\newcommand{\gA}{\mathfrak{A}}
\newcommand{\gU}{\mathfrak{U}}

\newcommand{\Sl}{\textup{SL}}
\newcommand{\Gl}{\textup{GL}}
\newcommand{\norm}{\textup{N}}
\newcommand{\trace}{\textup{Tr}}
\newcommand{\cf}{\textup{CF}}
\newcommand{\dir}{\textup{Dir}}

\begin{document}

\maketitle

\begin{abstract}
  Данная работа посвящена описанию группы симметрий многомерных цепных дробей. В качестве многомерного обобщения цепных дробей мы рассматриваем полиэдры Клейна. Мы выделяем два типа симметрий: симметрии Дирихле, соответствующие умножению на единицы соответствующего расширения поля $\Q$, и так называемые палиндромические симметрии. Основным результатом работы является критерий наличия у двумерной цепной дроби палиндромических симметрий, аналогичный известному критерию симметричности периода цепной дроби квадратичной иррациональности.
\end{abstract}

\section{Введение}\label{intro_ch}

Классическая теорема Лагранжа утверждает, что цепная дробь действительного числа $\alpha$ периодична тогда и только тогда, когда $\alpha$ является квадратичной иррациональностью. При этом период, прочитанный в обратном порядке, становится периодом цепной дроби сопряжённого числа. Это следует из теоремы Галуа, которую он доказал в своей самой первой работе  \cite{galois}. А именно, он показал, что если
\[\alpha=[\overline{a_0;a_1,\ldots,a_t}]\]
и $\alpha'$ --- сопряжённая к $\alpha$ квадратичная иррациональность, то
\[-1/\alpha'=[\overline{a_t;a_{t-1},\ldots,a_0}].\]
В частности, если для такого $\alpha$ слово $(a_0,\ldots,a_t)$ симметрично, то $\norm_{\Q(\alpha)/\Q}(\alpha)=\alpha\alpha'=-1$. Возникает естественный вопрос, каков критерий того, что период цепной дроби квадратичной иррациональности симметричен. При этом понятно, что само понятие симметричности периода требует уточнения, ибо, скажем, у последовательности с периодом $(1,2)$ слово $(2,1)$ также является периодом, но ни одно из этих слов не симметрично. Для этих целей очень уместным оказывается понятие циклического палиндрома. Напомним, что конечная последовательность $(a_1,a_2,\ldots,a_{t-1},a_t)$ называется \emph{циклическим палиндромом}, если существует такой циклический сдвиг индексов $\sigma$, что $a_k=a_{\sigma(t+1-k)}$ для любого $k\in\{1,\ldots,t\}$. Соответственно, более строго сформулировать вопрос о симметричности периода можно так: \textit{каков критерий того, что период цепной дроби квадратичной иррациональности является циклическим палиндромом?}

Ответ на этот вопрос следует из результатов Лежандра \cite{legendre} и Крайтчика \cite{kraitchik}, несколько переработанных Перроном в книге \cite{perron_book}. Чтобы его сформулировать, напомним, что в теории цепных дробей два числа $\alpha$ и $\omega$ называют \emph{эквивалентными} (и пишут $\alpha\sim\omega$), если существуют такие $a,b,c,d\in\Z$, что $\alpha=(a\omega + b)/(c\omega + d)$, $ad - bc = \pm 1$. Эквивалентность равносильна тому, что «хвосты»  цепных дробей чисел $\alpha$ и $\omega$ совпадают.

Будем обозначать через $\norm(\alpha)$ и $\trace(\alpha)$ соответственно норму $\norm_{\Q(\alpha)/\Q}(\alpha)$ и след $\trace_{\Q(\alpha)/\Q}(\alpha)$ алгебраического числа $\alpha$.

\begin{proposition}\label{two_dimension}
  Период цепной дроби квадратичной иррациональности $\alpha$ является циклическим палиндромом тогда и только тогда, когда выполняется хотя бы одно из следующих условий:

  \textup{(а)}\ $\alpha\sim\omega:\ \trace(\omega)=0\ (\text{эквивалентно тому, что } \, \omega^2\in\Q)$;

  \textup{(б)}\ $\alpha\sim\omega:\ \trace(\omega)=1\ (\text{эквивалентно тому, что } \, (\omega-1/2)^2\in\Q)$;

  \textup{(в)}\ $\alpha\sim\omega:\hskip 2.25mm \norm(\omega)=1$;

  \hskip 0.18mm
  \textup{(г)}\ $\alpha\sim\omega:\hskip 2.25mm \norm(\omega)=-1$.
  \\
  Более того, условие $(\textup{б})$ эквивалентно условию $(\textup{в})$.
\end{proposition}

Подробное доказательство предложения \ref{two_dimension}, использующее методы геометрии чисел, можно найти в статье \cite{german_tlyustangelov_mjcnt}. Данная же работа посвящена двумерному обобщению предложения \ref{two_dimension}.

Текст имеет следующую структуру: в параграфе \ref{sec_2} мы напоминаем геометрическую интерпретацию цепных дробей и описываем соответствующее многомерное обобщение; параграф \ref{sec_3} посвящён формулировке основного результата; в параграфе \ref{sec_4} мы изучаем симметрии, существование которых следует из теоремы Дирихле об алгебраических единицах; в параграфе \ref{sec_5} мы изучаем палиндромические симметрии двумерных цепных дробей; в параграфе \ref{sec:proof} мы доказываем основной результат данной работы; наконец, в параграфе \ref{sec_6} мы анализируем различие между двумерным результатом и одномерным.

\section{Геометрия цепных дробей}\label{sec_2}

\subsection{Полигоны Клейна и геометрические цепные дроби}\label{cpt_1}

Цепные дроби имеют довольно изящную геометрическую интерпретацию. Подробно эта интерпретация описана, например, в работах \cite{german_tlyustangelov_mjcnt}, \cite{korkina_2dim} (см. также книгу \cite{karpenkov_book}). Там же можно найти доказательства фактов, приводимых в данном пункте.

Рассмотрим на плоскости $\R^2$ прямые $l_1$ и $l_2$, проходящие через начало координат $\vec 0$, порождённые двумя различными векторами $(1, \alpha)$ и $(1, \beta)$ соответственно. Эти прямые делят плоскость на четыре угла $C_1$, $C_2$, $C_3$, $C_4$. Для каждого $i=1,2,3,4$ выпуклая оболочка
\[
  \cK(C_i)=\conv(C_i\cap\Z^2\setminus\{\vec 0\})
\]
ненулевых точек решётки $\Z^2$, лежащих в $C_i$, называется \emph{полигоном Клейна}. Граница $\partial(\cK(C_i))$ называется \emph{парусом Клейна}. Каждый парус --- неограниченная в обе стороны ломаная с целыми вершинами. \emph{Геометрической цепной дробью} $\cf(l_1,l_2)$ называется объединение всех четырёх парусов, соответствующих данной паре прямых.

Последовательности целочисленных длин рёбер полигонов Клейна и целочисленных углов между соседними рёбрами тесно связаны с неполными частными классической цепной дроби. Например, если $\alpha$ и $\beta$ иррациональны и выполняются условия $\alpha>1$ и $-1<\beta<0$, эти величины в точности равны соответствующим неполным частным чисел $\alpha$ и $-1/\beta$ (см. рис. \ref{fig:KP_and_CF}).

\begin{figure}[h]
  \centering
  \begin{tikzpicture}[scale=1.2]
    \draw[very thin,color=gray,scale=1] (-3.8,-1.7) grid (4.8,4.56);

    \draw[color=black] plot[domain=-13/9:3.6] (\x, {11*\x/8}) node[right]{$l_1$};
    \draw[color=black] plot[domain=-4:5] (\x, {-3*\x/8}) node[right]{$l_2$};

    \fill[blue!10!,path fading=north]
        (4.8,4.56) -- (3+0.4,4+0.4*7/5) -- (3,4) -- (1,1) -- (1,0) -- (4.8,4.56) -- cycle;
    \fill[blue!10!,path fading=east]
        (4.8,4.56) -- (1,0) -- (3,-1) -- (3+1.8,-1-1.8*2/5) -- cycle;
    \fill[blue!10!,path fading=north]
        (2+1.17,3+1.17*4/3) -- (2,3) -- (0,1) -- (-2-1.8,3+1.17*4/3) -- cycle;
    \fill[blue!10!,path fading=west]
        (-2-1.8,3+1.17*4/3) -- (0,1) -- (-2,1) -- (-2-1.8,1+1.8/3) -- cycle;
    \fill[blue!10!,path fading=south]
        (-0.7,-1.7) -- (0,-1) -- (2,-1) -- (4.1,-1.7) -- cycle;
    \fill[blue!10!,path fading=west]
        (-3.8,-1.7) -- (-1,0) -- (-3,1) -- (-3.8,1+0.8*2/5) -- cycle;
    \fill[blue!10!,path fading=south]
        (-1-0.7*2/3,-1.7) -- (-1,-1) -- (-1,0) -- (-3.8,-1.7) -- cycle;

    \draw[color=blue] (3+0.4,4+0.4*7/5) -- (3,4) -- (1,1) -- (1,0) -- (3,-1) -- (3+1.8,-1-1.8*2/5);
    \draw[color=blue] (2+1.17,3+1.17*4/3) -- (2,3) -- (0,1) -- (-2,1) -- (-2-1.8,1+1.8/3);
    \draw[color=blue] (-0.7,-1.7) -- (0,-1) -- (2,-1) -- (4.1,-1.7);
    \draw[color=blue] (-1-0.7*2/3,-1.7) -- (-1,-1) -- (-1,0) -- (-3,1) -- (-3.8,1+0.8*2/5);

    \node[fill=blue,circle,inner sep=1.2pt] at (3,4) {};
    \node[fill=blue,circle,inner sep=1.2pt] at (1,1) {};
    \node[fill=blue,circle,inner sep=1.2pt] at (1,0) {};
    \node[fill=blue,circle,inner sep=1.2pt] at (3,-1) {};
    \node[fill=blue,circle,inner sep=1.2pt] at (2,3) {};
    \node[fill=blue,circle,inner sep=1.2pt] at (0,1) {};
    \node[fill=blue,circle,inner sep=1.2pt] at (-2,1) {};

    \node[fill=blue,circle,inner sep=1.2pt] at (-1,-1) {};
    \node[fill=blue,circle,inner sep=1.2pt] at (-1,0) {};
    \node[fill=blue,circle,inner sep=1.2pt] at (-3,1) {};
    \node[fill=blue,circle,inner sep=1.2pt] at (0,-1) {};
    \node[fill=blue,circle,inner sep=1.2pt] at (2,-1) {};

    \node[fill=blue,circle,inner sep=0.8pt] at (1,-1) {};
    \node[fill=blue,circle,inner sep=0.8pt] at (-1,1) {};
    \node[fill=blue,circle,inner sep=0.8pt] at (1,2) {};

    \node[right] at (1-0.03,0.5) {$a_0$};
    \node[right] at (2-2/13,2.5-3/13) {$a_2$};

    \node[right] at (2-0.06,-0.4) {$a_{-2}$};

    \node[above left] at (1.08,2-0.02) {$a_1$};

    \node[above] at (-1,0.95) {$a_{-1}$};

    \draw[blue] ([shift=({atan(-1/2)}:0.2)]1,0) arc (atan(-1/2):90:0.2);
    \draw[blue] ([shift=({atan(-2/5)}:0.2)]3,-1) arc (atan(-2/5):90+atan(2):0.2);
    \draw[blue] ([shift=(-90:0.2)]1,1) arc (-90:atan(3/2):0.2);
    \draw[blue] ([shift=({-90-atan(2/3)}:0.2)]3,4) arc (-90-atan(2/3):atan(7/5):0.2);
    \draw[blue] ([shift=({atan(4/3)}:0.2)]2,3) arc (atan(4/3):225:0.2);
    \draw[blue] ([shift=(45:0.2)]0,1) arc (45:180:0.2);
    \draw[blue] ([shift=(0:0.2)]-2,1) arc (0:90+atan(3):0.2);

    \node[right] at (3.15,-0.88) {$a_{-3}$};
    \node[right] at (1.17,0) {$a_{-1}$};
    \node[right] at (1.17,1) {$a_1$};
    \node[right] at (3.08,3.75) {$a_3$};
    \node[left] at (2-0.1,3.23) {$a_2$};
    \node[above] at (-0.05,1.18) {$a_0$};
    \node[above] at (-2.02,1.16) {$a_{-2}$};

    \draw (3,1.22) node[right]{$\cK(C_1)$};
    \draw (-1,3.22) node[right]{$\cK(C_2)$};
  \end{tikzpicture}
  \caption{Полигоны Клейна для}
          {$\alpha=[a_0;a_1,a_2,\ldots],\ -1/\beta=[a_{-1};a_{-2},a_{-3},\ldots]$}
  \label{fig:KP_and_CF}
\end{figure}

Алгебраические числа степени $2$, они же квадратичные иррациональности, порождают полигоны Клейна, обладающие периодической структурой. Следующее утверждение представляет собой обобщение теоремы о существовании нетривиального решения уравнения Пелля и позволяет проинтерпретировать теорему Лагранжа геометрически (подробнее об этом см. в \cite{german_tlyustangelov_mjcnt}).

\begin{proposition}\label{prop:more_than_pelle}
  Пусть $\alpha$, $\beta$ --- различные иррациональные числа. Тогда следующие два утверждения эквивалентны:

  \textup{(а)} $\alpha$ и $\beta$ --- сопряжённые алгебраические числа степени $2$;

  \textup{(б)} $(1,\alpha)$ и $(1,\beta)$ --- собственные вектора некоторого неединичного оператора из $\Sl_2(\Z)$.
\end{proposition}

Ясно, что если $A$ --- оператор из предложения \ref{prop:more_than_pelle}, то и он, и все его степени сохраняют геометрическую цепную дробь $\cf(l_1,l_2)$ (то есть отображают объединение всех четырёх парусов на себя). При этом $A(l_1)=l_1$, $A(l_2)=l_2$. Соответственно, геометрический смысл палиндромичности периода цепной дроби квадратичной иррациональности описывается следующим утверждением.

\begin{proposition}\label{prop:geometry_of_palindromicity}
  Пусть $\alpha$ и $\beta$ --- сопряжённые алгебраические числа степени $2$. Пусть, как и прежде, $l_1$ и $l_2$ порождены векторами $(1, \alpha)$ и $(1, \beta)$. Тогда следующие два утверждения эквивалентны:

  \textup{(а)} период цепной дроби числа $\alpha$ является циклическим палиндромом;

  \textup{(б)} существует оператор $G\in\Gl_2(\Z)$, который сохраняет геометрическую цепную дробь $\cf(l_1,l_2)$, но для которого подпространства $l_1$, $l_2$ не являются собственными.

  В частности, для такого оператора $G$ справедливо $G(l_1)=l_2$, $G(l_2)=l_1$.
\end{proposition}

\subsection{Многомерные цепные дроби} \label{sub_2}

Конструкция, описанная в предыдущем пункте, имеет естественное многомерное обобщение. Пусть $l_1,\ldots,l_n$ --- одномерные подпространства пространства $\R^n$, линейная оболочка которых совпадает со всем $\R^n$. Тогда гиперпространства, натянутые на всевозможные $(n-1)$-наборы из этих подпространств, разбивают $\R^n$ на $2^n$ симплициальных конусов. Будем обозначать множество этих конусов через
\[ \mathcal{C}(l_1, \ldots, l_n).\]
Симплициальный конус с вершиной в начале координат $\vec{0}$ будем называть \emph{иррациональным}, если линейная оболочка любой его гиперграни не содержит целых точек, кроме начала координат  $\vec{0}$.

\begin{definition}
  Пусть  $C$ --- иррациональный конус, $C \in \mathcal{C}(l_1, \ldots, l_n)$. Выпуклая оболочка $\cK(C) = \conv(C\cap\Z^{n}\setminus\{\vec{0}\} )$ и его граница  $\partial(\cK(C))$ называются соответственно \emph{полиэдром Клейна} и \emph{парусом Клейна}, соответствующими конусу $C$. Объединение же парусов
  \[\cf(l_1, \ldots, l_n) = {\underset{C \, \in \, \mathcal{C}(l_1, \ldots, l_n)}{\bigcup}} \partial(\cK(C))\]
  называется \emph{$(n-1)$-мерной геометрической цепной дробью}.
\end{definition}

В иррациональном случае каждый полиэдр Клейна $\cK(C)$ является обобщённым многогранником, то есть его пересечение с любым компактным многогранником --- также компактный многогранник (доказательство см. в \cite{moussafir}). В этом случае каждый парус $\partial(\cK(C))$ является полиэдральной поверхностью, гомеоморфной $\R^{n-1}$ и состоящей из $(n-1)$-мерных многогранников, некоторые из которых, вообще говоря, могут быть неограниченными (отсутствие неограниченных граней равносильно иррациональности двойственного к $C$ конуса --- подробнее см. в \cite{german_matsbornik}). Грани полиэдра Клейна, будучи многомерными аналогами рёбер полигонов Клейна, играют роль неполных частных. Эта аналогия проявилась в результатах, полученных в работах \cite{german_matsbornik}, \cite{german_bordeaux}, \cite{german_lak_lagrange} (см. также замечательную книгу \cite{karpenkov_book}).

Про обыкновенные цепные дроби алгебраических чисел степени $n\geq3$ известно не много, не известно даже, могут ли неполные частные такого числа быть чем-то ограничены. Геометрические же цепные дроби, строящиеся по алгебраическим числам, обладают периодической структурой --- так же, как полигоны Клейна, строящиеся по квадратичным иррациональностям. Дело в том, что алгебраические числа степени $n$ тесно связаны с операторами из $\Gl_{n}(\Z)$. Например, справедливо следующее утверждение, обобщающее предложение \ref{prop:more_than_pelle}. Напомним, что оператор из $\Gl_{n}(\Z)$ с вещественными собственными значениями, характеристический многочлен которого неприводим над $\Q$, называется \emph{гиперболическим}.

\begin{proposition}\label{prop:more_than_pelle_n_dim}
  Числа $1,\alpha_1,\ldots,\alpha_{n-1}$ образуют базис некоторого вполне вещественного расширения $K$ поля $\Q$ тогда и только тогда, когда вектор $(1,\alpha_1,\ldots,\alpha_{n-1})$ является собственным для некоторого гиперболического оператора $A\in\Sl_n(\Z)$.
  При этом вектора $(1,\sigma_i(\alpha_1),\ldots,\sigma_i(\alpha_{n-1}))$, $i=1,\ldots,n$, где $\sigma_1(=\id),\sigma_2,\ldots,\sigma_n$ --- все вложения $K$ в $\R$, образуют собственный базис оператора $A$.
\end{proposition}

Предложение \ref{prop:more_than_pelle_n_dim}, равно как и предложение \ref{prop:more_than_pelle}, достаточно известно. Оно также следует из доказываемых нами в параграфе \ref{sec_4} лемм \ref{l:eigenvector_produces_a_field_basis} и \ref{l:field_basis_produces_an_eigenvector}.

\begin{definition}
  Пусть $l_1,\ldots,l_n$ --- собственные подпространства некоторого гиперболического оператора $A\in\Gl_n(\Z)$. Тогда $(n-1)$-мерная геометрическая цепная дробь $\cf(l_1,\ldots,l_n)$ называется \emph{алгебраической}. Мы будем также говорить, что эта дробь \emph{ассоциирована} с оператором $A$ и писать $\cf(A)=\cf(l_1,\ldots,l_n)$. Множество всех $(n-1)$-мерных алгебраических цепных дробей будем обозначать $\gA_{n-1}$.
\end{definition}

Ясно, что если $\cf(l_1,\ldots,l_n)=\cf(A)\in\gA_{n-1}$, то оператор $A$, а также любая его степень, сохраняет цепную дробь $\cf(l_1,\ldots,l_n)$ (то есть отображает объединение всех $2^n$ парусов на себя). Из теоремы Дирихле об алгебраических единицах следует, что существует подгруппа группы $\Gl_{n}(\Z)$, изоморфная $(\Z/2\Z)\times\Z^{n-1}$, каждый элемент которой коммутирует с $A$ и сохраняет $\cf(l_1,\ldots,l_n)$. Доказательство этого факта для полноты изложения мы приводим в параграфе \ref{sec_4}. Целью же данной работы является поиск критерия существования дополнительных симметрий дроби $\cf(l_1,\ldots,l_n)$ в случае $n=3$.

\section{Симметрии многомерных алгебраических цепных дробей и формулировка основного результата}\label{sec_3}

Пусть $A\in\Gl_n(\Z)$ --- гиперболический оператор и $\cf(A)=\cf(l_1, \ldots, l_n)$ --- ассоциированная с ним $(n-1)$-мерная алгебраическая цепная дробь. Будем называть \emph{группой симметрий} этой цепной дроби множество
\[
  \textup{Sym}_{\Z}\big(\cf(A)\big)=
  \Big\{ G\in\Gl_n(\Z) \ \Big|\ G\big(\cf(A)\big)=\cf(A) \Big\}.
\]
Из соображений непрерывности ясно, что условие $G\big(\cf(A)\big)=\cf(A)$ равносильно тому, что $G(l_1 \cup \ldots \cup l_n\big)=l_1\cup\ldots\cup l_n$. Следовательно, для каждого $G\in\textup{Sym}_{\Z}\big(\cf(A)\big)$ однозначно определена перестановка $\sigma_G$, такая что
\begin{equation} \label{eq:repres}
  G(l_{i})=l_{\sigma_G(i)},\quad i=1,\dots,n.
\end{equation}
И обратно, если для $G\in\Gl_n(\Z)$ существует такая перестановка $\sigma_{G}$, что выполняются соотношения \ref{eq:repres}, то $G\in\textup{Sym}_{\Z}\big(\cf(A)\big)$.

\begin{definition}
  Оператор $G\in\textup{Sym}_{\Z}\big(\cf(A)\big)$, такой что $\sigma_G=\id$, будем называть \emph{симметрией Дирихле} дроби $\cf(A)\in\gA_{n-1}$. Группу всех симметрий Дирихле будем называть \emph{группой Дирихле} оператора $A$ и обозначать $\dir(A)$.
\end{definition}

\begin{definition}
  Оператор $G\in\textup{Sym}_{\Z}\big(\cf(A)\big)$, не являющийся симметрией Дирихле, будем называть \emph{палиндромической симметрией} дроби $\cf(A)$. Если множество палиндромических симметрией цепной дроби непусто, то такую цепную дробь будем называть \emph{палиндромичной}.
\end{definition}

Предложение \ref{prop:geometry_of_palindromicity} из пункта \ref{cpt_1} позволяет переформулировать предложение \ref{two_dimension} как критерий палиндромичности (одномерной) геометрической цепной дроби. Для формулировки основного результата данной работы --- аналогичного критерия для двумерной геометрической цепной дроби, остаётся дать следующее определение, естественным образом обобщающее понятие эквивалентности (обыкновенных) цепных дробей.

\begin{definition}\label{def:equivalence}
  Пусть $\vec v_1$, $\vec v_2$ --- вектора в $\R^n$ и пусть их первые координаты равны $1$. Будем говорить, что $\vec v_1$ и $\vec v_2$ \emph{эквивалентны} и писать $\vec v_1\sim\vec v_2$, если существует такой оператор $X\in\Gl_n(\Z)$ и такое $\mu\in\R$, что $X\vec v_1=\mu\vec v_2$.
\end{definition}

\begin{remark}
  В случае $n=2$ отношение $(1,\alpha)\sim(1,\beta)$ равносильно отношению $\alpha\sim\beta$ из параграфа \ref{intro_ch}.
\end{remark}

Следующая теорема является основным результатом данной работы.

\begin{theorem}\label{main_t_1}
  Пусть $\cf(l_1,l_2,l_3)\in\gA_2$ и пусть подпространство $l_1$ порождено вектором $(1, \alpha, \beta)$. Тогда $\cf(l_1, l_2, l_3)$ палиндромична в том и только в том случае, если существует такое алгебраическое число $\omega$ степени $3$ со своими сопряжёнными $\omega'$ и $\omega''$, что выполнено хотя бы одно из следующих условий:

  \textup{(а)} $(1, \alpha, \beta)\sim(1, \omega, \omega'):\hskip 6.5mm \trace(\omega)=\omega + \omega' + \omega'' =0$;

  \textup{(б)} $(1, \alpha, \beta)\sim(1, \omega, \omega'):\hskip 6.5mm \trace(\omega)=\omega + \omega' + \omega'' =1$.
  \\
  Более того, существование $\omega$, удовлетворяющего условию $(\textup{б})$, равносильно существованию $\omega$, удовлетворяющего любому из следующих двух условий:

  \textup{(в)} $(1, \alpha, \beta)\sim(1, \omega, 1/\omega'):\hskip 3.98mm \norm(\omega)=\omega  \omega'  \omega''=1$;

  \hskip 0.18mm
  \textup{(г)} $(1, \alpha, \beta)\sim(1, \omega, -1/\omega'):\ \norm(\omega)=\omega  \omega'  \omega''=-1$.
\end{theorem}

Стоит отметить следующее наблюдение, непосредственно вытекающее из теоремы \ref{main_t_1}:

\begin{corollary}
  Если дробь $\cf(l_1, l_2, l_3)\in\gA_2$ палиндромична и подпространство $l_1$ порождено вектором $(1, \alpha, \beta)$, то (кубическое) расширение $\Q(\alpha)$ поля $\Q$ является нормальным.
\end{corollary}

\begin{remark}
  Отметим также, что теорема \ref{main_t_1} представляет собой результат о симметрии алгебраических конусов. По этой причине аналогичное утверждение справедливо не только для полиэдров Клейна, но и для других геометрических конструкций, обобщающих обыкновенные цепные дроби. Например, для \emph{трёхмерной цепной дроби Минковского--Вороного}. При заданных прямых $l_1,l_2,l_3$ такая дробь определяется как граница объединения всех $\vec 0$-сим\-мет\-рич\-ных параллелепипедов с рёбрами параллельными этим прямым, не содержащих в своей внутренности ненулевых целых точек. Подробнее о таких дробях и о соответствующих \emph{комплексах Минковского--Вороного} см. работы \cite{ustinov_UMN_2015}, \cite{ustinov_karpenkov_JNT_2017}.
\end{remark}

\section{Симметрии Дирихле} \label{sec_4}

Несложно показать (см., например, \cite{korkina_2dim}), что группа $\dir(A)$ гиперболического оператора $A\in\Gl_n(\Z)$ совпадает с множеством операторов из $\Gl_{n}(\Z)$, представимых как многочлены с рациональными коэффициентами от $A$. Более явное, на наш взгляд, описание этой группы даёт теорема Дирихле об алгебраических единицах.

Напомним, что модуль $M$, содержащийся в конечном расширении $K$ поля $\Q$, называется \emph{полным}, если его ранг максимален, то есть равен $[K:\Q]$. Если $M$ --- полный модуль в поле $K$, то группа
\[
  \gU_M=\big\{ \e\in K \,\big|\, \e M=M \big\}
\]
называется \emph{группой единиц} модуля $M$. Структура группы $\gU_M$ описывается теоремой Дирихле об алгебраических единицах. Нам понадобится следующий её частный случай для вполне вещественных расширений поля $\Q$ (подробнее см., например, в \cite{borevich_shafarevich}).

\vskip 2.2mm
\textsc{Теорема Дирихле.}
  \textit{Пусть $K$ --- вполне вещественное расширение поля $\Q$ степени $n$ и пусть $M$ --- произвольный полный модуль в $K$. Тогда существует такой набор единиц $\e_1,\ldots,\e_{n-1}\in\gU_M$, что любая единица $\e\in\gU_M$ однозначно представляется в виде
  \[\e=\zeta\e_1^{z_1}\ldots\e_{n-1}^{z_{n-1}},\]
  где $z_1,\dots,z_{n-1}\in\Z$ и $\zeta\in\{-1,1\}$. В частности, $\gU_M\cong(\Z/2\Z)\times\Z^{n-1}$.}


\begin{remark}
  Обычно теорема Дирихле формулируется для групп единиц порядков поля $K$ (\emph{порядком} называется полный модуль поля $K$, являющийся кольцом и содержащий число $1$). Но если $M$ --- произвольный полный модуль в $K$, его группа единиц $\gU_M$ в точности совпадает с мультипликативной группой $\gO_M^\times$ \emph{кольца множителей} $\gO_M$ модуля $M$, состоящего из таких чисел $\e\in K$, что $\e M\subset M$. Ясно, что $\gO_M$ является порядком, и, обратно, любой порядок поля $K$ является кольцом множителей для самого себя. Поэтому приведённая формулировка теоремы Дирихле эквивалентна стандартной.
\end{remark}

\begin{lemma}\label{l:eigenvector_produces_a_field_basis}
  Пусть $(1,\alpha_1,\ldots,\alpha_{n-1})$ --- собственный вектор гиперболического оператора $A\in\Gl_n(\Z)$, соответствующий собственному значению $\lambda$. Тогда числа $1,\alpha_1,\ldots,\alpha_{n-1}$ образуют базис поля $\Q(\lambda)$ над $\Q$.
\end{lemma}

\begin{proof}
  Положим $K=\Q(\lambda)$. Из гиперболичности оператора $A$ следует, что $K$ --- вполне вещественное расширение поля $\Q$ степени $n$. При этом числа $\alpha_1,\ldots,\alpha_{n-1}$ содержатся в $K$, ибо являются решением системы линейных уравнений с коэффициентами из этого поля. Предположим, что набор $1,\alpha_1,\dots,\alpha_{n-1}$ не является базисом поля $K$. Тогда они линейно зависимы над $\Q$. Не ограничивая общности, можно считать, что $\alpha_{n-1}$ --- линейная комбинация чисел $1,\alpha_1,\dots,\alpha_{n-2}$ с рациональными коэффициентами. Тогда существует такая матрица $C\in\textup{Mat}_{n\times n}(\Q)$ с нулевым последним столбцом, что $\det(C-\lambda I_n)=0$. Последнее противоречит неприводимости характеристического многочлена оператора $A$. Следовательно, набор $1,\alpha_1,\dots,\alpha_{n-1}$ является базисом поля $K$.
\end{proof}

\begin{lemma}\label{l:eigenvalue_is_a_unit}
  Пусть $B$ --- произвольный оператор из $\Gl_n(\Z)$ и пусть вектор $(1,\alpha_1,\ldots,\alpha_{n-1})$ является собственным вектором $B$, соответствующим собственному значению $\lambda$. Пусть числа $1,\alpha_1,\ldots,\alpha_{n-1}$ линейно независимы над $\Q$. Тогда $\lambda$ является единицей модуля $M=\Z+\Z\alpha_1+\dots +\Z\alpha_{n-1}$.
\end{lemma}

\begin{proof}
  Поскольку числа $1,\alpha_1,\dots,\alpha_{n-1}$ линейно независимы над $\Q$, они образуют базис модуля $M$. Поскольку $B(M^n)=M^n$ и $(1,\alpha_1,\dots,\alpha_{n-1})\in M^n$, справедливо $(\lambda,\lambda\alpha_1,\dots,\lambda\alpha_{n-1})\in M^n$. В силу обратимости оператора $B$ числа $\lambda,\lambda\alpha_1,\dots,\lambda\alpha_{n-1}$ также образуют базис $M$. Стало быть, $\lambda\in\gU_M$.
\end{proof}

\begin{lemma}\label{l:field_basis_produces_an_eigenvector}
  Пусть числа $1,\alpha_1,\ldots,\alpha_{n-1}$ образуют базис некоторого 
  расширения $K$ поля $\Q$ и пусть $\e$ --- единица модуля $M=\Z+\Z\alpha_1+\dots +\Z\alpha_{n-1}$. Тогда существует единственный оператор $B\in\Gl_n(\Z)$, такой что $\e$ является собственным значением $B$, а вектор $(1,\alpha_1,\ldots,\alpha_{n-1})$ является собственным вектором оператора $B$, соответствующим $\e$. При этом $\det B=\norm_{K/\Q}(\e)$, а вектора $(1,\sigma_i(\alpha_1),\ldots,\sigma_i(\alpha_{n-1}))$, $i=1,\ldots,n$, где $\sigma_1(=\id),\sigma_2,\ldots,\sigma_n$ --- все вложения $K$ в $\Comp$, образуют собственный базис оператора $B$.
\end{lemma}

\begin{proof}
   Модуль $M$ является полным, а числа $1,\alpha_1,\dots,\alpha_{n-1}$ образуют его базис. Поскольку $\e\in\gU_M$, числа $\e,\e\alpha_1,\dots,\e\alpha_{n-1}$ также образуют базис $M$. Следовательно, существует ровно одна матрица $B\in\Gl_n(\Z)$, такая что
  \[
    (\e,\e\alpha_1,\dots,\e\alpha_{n-1})^\top=B(1,\alpha_1,\dots,\alpha_{n-1})^\top.
  \]
  Применяя к обеим частям этого равенства $\sigma_i$, получаем
  \[
    \big(\sigma_i(\e),\sigma_i(\e)\sigma_i(\alpha_1),\dots,\sigma_i(\e)\sigma_i(\alpha_{n-1})\big){}^\top=B\big(1,\sigma_i(\alpha_1),\dots,\sigma_i(\alpha_{n-1})\big){}^\top.
  \]
  При этом вектора $(1,\sigma_i(\alpha_1),\ldots,\sigma_i(\alpha_{n-1}))$, $i=1,\ldots,n$, линейно независимы, поскольку дискриминант $D_{K/\Q}(1,\alpha_1,\dots,\alpha_{n-1})$ отличен от нуля. Стало быть, они образуют собственный базис оператора $B$ и $\det B=\prod_{i=1}^{n}\sigma_i(\e)=\norm_{K/\Q}(\e)$.
\end{proof}

\begin{corollary}\label{cor:eigenvector_determines_all}
  Собственный базис произвольного гиперболического оператора из $\Gl_n(\Z)$ однозначно (с точностью до коэффициентов пропорциональности) определяется любым из его собственных векторов.
\end{corollary}

Из лемм \ref{l:eigenvector_produces_a_field_basis} и \ref{l:field_basis_produces_an_eigenvector} непосредственно выводится предложение \ref{prop:more_than_pelle_n_dim}, поскольку, как известно, в любом вещественном расширении поля $\Q$ существует примитивный элемент, являющийся единицей (верно даже большее: в качестве такой единицы можно выбрать некоторое число Пизо, см. теорему 5.2.2 в книге \cite{bertin_pisot_salem}). Основной же целью данного параграфа является доказательство следующего утверждения о структуре группы $\dir(A)$. Отметим, что это утверждение уточняет в чисто вещественном случае следствие 17.10 из книги \cite{karpenkov_book} (см. там же предложение 17.11 и предшествующее ему обсуждение).

\begin{proposition}
  Для любого гиперболического оператора $A\in\Gl_n(\Z)$ справедливо
  \[
    \dir(A)\cong(\Z/2\Z)\times\Z^{n-1}.
  \]
\end{proposition}

\begin{proof}
  Пусть $\cf(A)=\cf(l_1,\ldots,l_n)$ и пусть подпространство $l_1$ порождается вектором $\vec l_1=(1,\alpha_1,\ldots,\alpha_{n-1})$. Положим
  \[
    K=\Q(\alpha_1,\ldots,\alpha_{n-1}),
    \qquad
    M=\Z+\Z\alpha_1+\dots +\Z\alpha_{n-1}.
  \]
  По лемме \ref{l:eigenvector_produces_a_field_basis} числа $1,\alpha_1,\ldots,\alpha_{n-1}$ образуют базис поля $K$. Следовательно, $M$ --- полный модуль в $K$.

  Покажем, что $\dir(A)\cong\gU_M$. Для любого оператора $B\in\dir(A)$ вектор $\vec l_1$ является собственным, стало быть, по лемме \ref{l:eigenvalue_is_a_unit} собственное значение $\lambda(B,\vec l_1)$ оператора $B$, которому соответствует $\vec l_1$, принадлежит $\gU_M$. Следовательно, отображение
  \begin{align*}
    \varphi & : \dir(A)\to\gU_M \\
    \varphi & : B\mapsto\lambda(B,\vec l_1)
  \end{align*}
  корректно определено. По лемме \ref{l:field_basis_produces_an_eigenvector} оно взаимно однозначно. При этом $\varphi$, очевидно, является гомоморфизмом. Стало быть, действительно, $\dir(A)\cong\gU_M$.

  Остаётся применить теорему Дирихле.
\end{proof}

\section{Палиндромические симметрии}\label{sec_5}

Отныне будем считать, что $n=3$, то есть будем рассматривать двумерные цепные дроби. Напомним, что множество всех двумерных алгебраических цепных дробей мы обозначаем через $\gA_2$.

Напомним также, что для каждого $G\in\textup{Sym}_{\Z}\big(\cf(A)\big)$ соотношением \ref{eq:repres} определена перестановка $\sigma_G$.

\begin{lemma}\label{ord_3}
  Пусть $G$ --- палиндромическая симметрия дроби $\cf(l_1,l_2,l_3)\in\gA_2$, ассоциированной с (гиперболическим) оператором $A$.
  Тогда

  \textup{1)} $\textup{ord}({\sigma_{G}}) = 3$;

  \textup{2)} $G^3 = \pm I_3$ (где $I_3$ --- единичный оператор);

  \textup{3)} существуют одномерное и двумерное рациональные подпространства $l$ и $\pi$, такие что $Gl=l$, $G\pi = \pi$ и $l + \pi = \R^3$.
\end{lemma}

\begin{proof}
  Случай $\textup{ord}(\sigma_{G}) = 1$ невозможен в силу того, что оператор $G$ не является симметрией Дирихле $\cf(A)$.

  Предположим $\textup{ord}(\sigma_{G}) = 2$. Пусть $\vec l_1,\vec l_2,\vec l_3$ --- произвольные вектора, порождающие подпространства $l_1,l_2,l_3$ соответственно. Изменив при необходимости нумерацию подпространств $l_1, l_2, l_3$, можно считать, что существуют такие вещественные числа $\mu_{1}$, $\mu_{2}$, $\mu_{3}$, что матрица оператора $G$ в базисе $\vec{l}_{1}, \vec{l}_{2}, \vec{l}_{3}$ имеет вид
 \[
   \begin{pmatrix}
     \mu_{1} & 0 & 0 \\
     0 & 0 & \mu_{2} \\
     0 & \mu_{3} & 0
   \end{pmatrix}.
 \]
  Тогда характеристический многочлен оператора $G$ имеет вид
  \[\chi_{G}(x) = (x - \mu_{1})(x^2 - \mu_{2}\mu_{3}) = x^3 - \mu_{1} x^2 - \mu_{2}\mu_{3} x \pm 1 \in \Z[x].\]
  Следовательно, $\mu_{1}$ --- целое число, и при этом $\mu_{1}$ --- корень уравнения $\chi_{G}(x) = 0$, то есть $\mu_{1} = \pm 1$. Стало быть, $l_1$ --- собственное подпространство оператора $G$, соответствующее собственному значению $\mu_{1} = \pm 1$. То есть $l_1$ рационально, что противоречит гиперболичности оператора $A$. Таким образом $\textup{ord}(\sigma_{G}) = 3$.

  Изменив при необходимости нумерацию подпространств $l_1, l_2, l_3$, можно считать, что существуют такие вещественные числа $\mu_{1}$, $\mu_{2}$, $\mu_{3}$, что матрица оператора $G$ в базисе $\vec{l}_{1}, \vec{l}_{2}, \vec{l}_{3}$ имеет вид
  \[
    \begin{pmatrix}
      0 & 0 & \mu_{1} \\
      \mu_{2} & 0 & 0 \\
      0 & \mu_{3} & 0
    \end{pmatrix}.
  \]
  Тогда характеристический многочлен имеет вид
  \[\chi_{G}(x) =  x^3 - \mu_{1}\mu_{2}\mu_{3} = x^3  - \det G.\]
  По теореме Гамильтона-Кэли $G^3 = (\det G)I_3 = \pm I_3$, что доказывает второе утверждение. Далее, поскольку одно из собственных значений равно $\det G$ (и равно $\pm 1$), а два оставшихся --- сопряжённые комплексные числа, у оператора $G$ есть одномерное собственное подпространство $l$ и двумерное инвариантное подпространство $\pi$. При этом, поскольку подпространство $l$ соответствует собственному значению  $\pm 1$, оно рационально. Кроме того, для любой точки $\vec{z} \in \Z^3 \setminus l$ три точки $\vec{z}$, $G^2(\vec{z})$ и $G^4(\vec{z})$ задают плоскость, параллельную $\pi$, следовательно, $\pi$ также рационально.
\end{proof}

\begin{corollary}\label{dif}
  Пусть $\cf(l_1,l_2,l_3)\in\gA_2$ и пусть $G$ --- палиндромическая симметрия $\cf(l_1,l_2,l_3)$.
  Положим
  \[
    G_+=(\det G)G,
    \qquad
    G_-=-(\det G)G.
  \]
  Тогда $G_+$ и $G_-$ --- также палиндромические симметрии $\cf(l_1,l_2,l_3)$ и существует конус $C \in \mathcal{C}(l_1, l_2, l_3)$, такой что $G_+(C)=C$ и $G_-(C)=-C$.
\end{corollary}

\begin{proof}
  Палиндромичность $G_+$ и $G_-$ очевидна. Пусть $l$ --- прямая из леммы \ref{ord_3}. Поскольку $l$ не содержится в $l_i+l_j$ для любых $i$ и $j$, найдётся такой конус $C\in\cC(l_1,l_2,l_3)$, что $l$ имеет непустое пересечение с его внутренностью. Поскольку $G(C)\in\cC(l_1,l_2,l_3)$, справедливо либо $G(C)=C$, либо $G(C)=-C$, в зависимости от того, $\det G=1$ или $-1$. Соответственно, $G_+(C)=C$ и $G_-(C)=-C$.
\end{proof}

\begin{lemma}\label{main_lem}
  Пусть $G$ --- палиндромическая симметрия дроби $\cf(l_1,l_2,l_3)\in\gA_2$. 
  Положим $F=G_+$ (см. следствие \ref{dif}). Тогда существуют $\vec{z}_1$, $\vec{z}_2$, $\vec{z}_3 \in \Z^3$, такие что
  \[F(\vec{z}_1) = \vec{z}_2,F(\vec{z}_2) = \vec{z}_3,F(\vec{z}_3) = \vec{z}_1\]
  и выполняется хотя бы одно из следующих двух утверждений:

  \textup{(а)} вектора $\vec z_1$, $\vec z_2$, $\vec w$, где $\vec{w}=\frac{1}{3}(\vec{z}_1+\vec{z}_2+\vec{z}_3)$, образуют базис решётки $\Z^3$;

  \textup{(б)} вектора $\vec z_1$, $\vec z_2$, $\vec z_3$ образуют базис решётки $\Z^3$.
\end{lemma}

\begin{proof}
  Пусть $l$ и $\pi$ --- одномерное и двумерное подпространства из леммы \ref{ord_3}. Рассмотрим ближайшую к $\pi$ рациональную плоскость $\pi_1$, параллельную $\pi$ и не совпадающую с $\pi$ (любую из двух). Тогда $F(\pi_1) = \pi_1$.

  Построим точки $\vec z_1,\vec z_2,\vec z_3\in\pi_1$ при помощи следующей итерационной процедуры. Возьмём произвольную целочисленную точку $\vec{v}_1 \in \pi_1 \setminus l$, и положим $\Delta_1 = \textup{conv}\big(\vec{v}_1, F(\vec{v}_1), F^2(\vec{v}_1)\big)$, $\vec{w} = \frac{1}{3}\big(\vec{v}_1 + F(\vec{v}_1) + F^2(\vec{v}_1)\big)$. Тогда $F(\Delta_1) = \Delta_1$ и $F\vec{w} = \vec{w} \in l$.

  Предположим, мы построили треугольник $\Delta_j$. Если в $\Delta_j$ есть целочисленная точка, отличная от вершин и $\vec{w}$, обозначим её через $\vec{v}_{j+1}$ и определим следующий треугольник как $\Delta_{j+1} = \textup{conv}\big(\vec{v}_{j+1}, F(\vec{v}_{j+1}), F^2(\vec{v}_{j+1})\big)$. Тогда $\Delta_{j+1}$ является собственным подмножеством $\Delta_j$ и при этом $\vec{w} = \frac{1}{3}\big(\vec{v}_{j+1} + F(\vec{v}_{j+1}) + F^2(\vec{v}_{j+1})\big)$.

  Последовательность $(\Delta_j)$ конечна. Пусть $\Delta_k$ --- последний её элемент. Положим $\vec{z}_1 = \vec{v}_k$, $\vec{z}_2 = F(\vec{v}_k)$, $\vec{z}_3 = F^2(\vec{v}_k)$. Напомним, что $\pi_1$ --- ближайшая к $\pi$ рациональная плоскость, параллельная $\pi$. Стало быть, если $\vec{w} \notin \Z^3$, то $\vec{z}_1$, $\vec{z}_2$, $\vec{z}_3$ --- базис $\Z^3$. Если же $\vec{w} \in \Z^3$, то $\vec{z}_1$, $\vec{z}_2$, $\vec{w}$ --- базис $\Z^3$.
\end{proof}

Если задана дробь $\cf(l_1,l_2,l_3)=\cf(A)\in\gA_2$, будем считать, что подпространство $l_1$ порождается вектором $\vec l_1=(1,\alpha,\beta)$. Тогда по лемме \ref{l:eigenvector_produces_a_field_basis} числа $1,\alpha,\beta$ образуют базис поля $K=\Q(\alpha,\beta)$ над $\Q$, а по лемме \ref{l:field_basis_produces_an_eigenvector} каждое $l_i$ порождается вектором $\vec l_i=(1,\sigma_i(\alpha),\sigma_i(\beta))$, где $\sigma_1(=\id),\sigma_2,\sigma_3$ --- все вложения $K$ в $\R$. То есть, если через $\big(\vec l_1,\vec l_2,\vec l_3\big)$ обозначить матрицу со столбцами $\vec l_1,\vec l_2,\vec l_3$, получим
\[
  \big(\vec l_1,\vec l_2,\vec l_3\big)=
  \begin{pmatrix}
    1 & 1 & 1 \\
    \alpha & \sigma_2(\alpha) & \sigma_3(\alpha) \\
    \beta & \sigma_2(\beta) & \sigma_3(\beta)
  \end{pmatrix}.
\]
Определим следующие классы двумерных алгебраических цепных дробей:
\begin{align*}
  \mathbf{CF}_1 & = \Big\{ \cf(l_1,l_2,l_3)\in\gA_2 \,\Big|\, \beta = \sigma_{2}(\alpha),\ \trace(\alpha)=0 \Big\}, \\
  \mathbf{CF}_2 & = \Big\{ \cf(l_1,l_2,l_3)\in\gA_2 \,\Big|\, \beta = \sigma_{2}(\alpha),\ \trace(\alpha)=1 \Big\}, \\
  \mathbf{CF}_3 & = \Big\{ \cf(l_1,l_2,l_3)\in\gA_2 \,\Big|\, \beta = \sigma_{2}(\alpha)^{-1},\ \norm(\alpha)=1 \Big\}, \\
  \mathbf{CF}_4 & = \Big\{ \cf(l_1,l_2,l_3)\in\gA_2 \,\Big|\, \beta = -\sigma_{2}(\alpha)^{-1},\ \norm(\alpha)=-1 \Big\}.
\end{align*}
Обозначим также для каждого $i=1,2,3,4$ через $\overline{\mathbf{CF}}_i$ образ $\mathbf{CF}_i$ при действии группы $\Gl_3(\Z)$:
\[
  \overline{\mathbf{CF}}_i=
  \Big\{ \cf(l_1,l_2,l_3)\in\gA_2 \,\Big|\, \exists X\in\Gl_3(\Z):X\big(\cf(l_1,l_2,l_3)\big)\in\mathbf{CF}_i \Big\}.
\]

\begin{lemma}\label{l:CF_instead_of_statements}
  Для дроби $\cf(l_1,l_2,l_3)\in\gA_2$ выполняется условие \textup{(а)}, \textup{(б)}, \textup{(в)} или \textup{(г)} теоремы \ref{main_t_1} тогда и только тогда, когда $\cf(l_1,l_2,l_3)$ принадлежит классу $\overline{\mathbf{CF}}_1$, $\overline{\mathbf{CF}}_2$, $\overline{\mathbf{CF}}_3$ или $\overline{\mathbf{CF}}_4$ соответственно.
\end{lemma}

\begin{proof}
  Для любого $X\in\Gl_3(\Z)$ гиперболичность оператора $A\in\Gl_3(\Z)$ равносильна гиперболичности оператора $XAX^{-1}$. При этом собственные подпространства гиперболического оператора однозначно восстанавливаются по любому его собственному вектору (см. следствие \ref{cor:eigenvector_determines_all}). Остаётся воспользоваться определением эквивалентности (см. определение \ref{def:equivalence}).
\end{proof}

Покажем, что все дроби из классов $\mathbf{CF}_1$, $\mathbf{CF}_2$, $\mathbf{CF}_3$, $\mathbf{CF}_4$ палиндромичны. Положим
\[
  F_{1} =
  \begin{pmatrix}
    1 & \phantom{-}0 & \phantom{-}0 \\
    0 & \phantom{-}0 & \phantom{-}1 \\
    0 & -1 & -1
  \end{pmatrix},
  \quad
  F_{2} =
  \begin{pmatrix}
    1 & \phantom{-}0 & \phantom{-}0 \\
    0 & \phantom{-}0 & \phantom{-}1 \\
    1 & -1 & -1
  \end{pmatrix},
  \phantom{-}
\]
\[
  F_{3} =
  \begin{pmatrix}
    0 & \phantom{-}0 & \phantom{-}1 \\
    1 & \phantom{-}0 & \phantom{-}0 \\
    0 & \phantom{-}1 & \phantom{-}0
  \end{pmatrix},
  \quad
  F_{4} =
  \begin{pmatrix}
    \phantom{-}0 & \phantom{-}0 & \phantom{-}1 \\
    -1 & \phantom{-}0 & \phantom{-}0 \\
    \phantom{-}0 & -1 & \phantom{-}0
  \end{pmatrix}.
\]

\begin{lemma}\label{oper_eq}
  Пусть $\cf(l_1,l_2,l_3)\in\gA_2$ и $i\in\{1,2,3,4\}$. Тогда $\cf(l_1,l_2,l_3)$ принадлежит классу $\mathbf{CF}_{i}$ в том и только в том случае, если $F_{i}$ --- её палиндромическая симметрия.
\end{lemma}

\begin{proof}
  Из леммы \ref{ord_3} следует, что оператор $F\in\Gl_3(\Z)$ является палиндромической симметрией дроби $\cf(l_1,l_2,l_3)$ тогда и только тогда, когда с точностью до перестановки индексов существуют такие действительные  числа $\mu_1,\mu_2,\mu_3$, что $F\big(\vec l_1,\vec l_2,\vec l_3\big)=\big(\mu_2\vec l_2,\mu_3\vec l_3,\mu_1\vec l_1\big)$.

  Пусть $\cf(l_1, l_2, l_3) \in \mathbf{CF}_{1}$. Тогда
  \[
    F_{1} \big(\vec{l}_1, \vec{l}_2, \vec{l}_3\big)=
    \begin{pmatrix}
      1 & 1 & 1 \\
      \sigma_{2}(\alpha) &  \sigma_{3}(\alpha) & \alpha\\
      -\alpha - \sigma_{2}(\alpha) & -\sigma_{2}(\alpha) - \sigma_{3}(\alpha) & -\sigma_{3}(\alpha)- \alpha\\
    \end{pmatrix}
    =\big(\vec{l}_2, \vec{l}_3, \vec{l}_1\big).
  \]
  Следовательно, $F_{1}$ ---  палиндромическая симметрия  $\cf(l_1,l_2,l_3)$. Обратно, предположим, $F_{1}$ ---  палиндромическая симметрия  $\cf(l_1,l_2,l_3)$. Тогда существует $\mu_2$, такое что с точностью до перестановки индексов
  \[
    F_{1}\vec{l}_1=
    \begin{pmatrix}
      1  \\
      \beta \\
      - \alpha - \beta
    \end{pmatrix}
    =\mu_2
    \begin{pmatrix}
      1  \\
      \sigma_{2}(\alpha) \\
      \sigma_{2}(\beta)
    \end{pmatrix},
  \]
  откуда $\mu_2 = 1$, $\beta =\sigma_{2}(\alpha)$ и $\trace(\alpha) = 0$. Стало быть, $\cf(l_1, l_2, l_3) \in \mathbf{CF}_{1}$.

  Пусть $\cf(l_1, l_2, l_3) \in \mathbf{CF}_{2}$. Тогда
  \[
    F_{2} \big(\vec{l}_1, \vec{l}_2, \vec{l}_3\big) =
    \begin{pmatrix}
      1 & 1 & 1 \\
      \sigma_{2}(\alpha) &  \sigma_{3}(\alpha) & \alpha\\
      1 - \alpha - \sigma_{2}(\alpha) & 1 - \sigma_{2}(\alpha) - \sigma_{3}(\alpha) & 1 - \sigma_{3}(\alpha) - \alpha\\
    \end{pmatrix}
    =\big(\vec{l}_2, \vec{l}_3, \vec{l}_1\big).
  \]
  Следовательно, $F_{2}$ ---  палиндромическая симметрия  $\cf(l_1, l_2, l_3)$. Обратно, предположим, $F_{2}$ ---  палиндромическая симметрия  $\cf(l_1, l_2, l_3)$. Тогда существует $\mu_2$, такое что с точностью до перестановки индексов
  \[
    F_{2}\vec{l}_1 =
    \begin{pmatrix}
      1  \\
      \beta \\
      1- \alpha - \beta
    \end{pmatrix}
    =\mu_2
    \begin{pmatrix}
      1  \\
      \sigma_{2}(\alpha) \\
      \sigma_{2}(\beta)
    \end{pmatrix},
  \]
  откуда $\mu_2 = 1$, $\beta =\sigma_{2}(\alpha)$ и $\trace(\alpha) = 1$. Стало быть, $\cf(l_1, l_2, l_3) \in \mathbf{CF}_{2}$.

  Пусть $\cf(l_1, l_2, l_3) \in \mathbf{CF}_{3}$. Тогда
  \[
    F_{3} \big(\vec{l}_1, \vec{l}_2, \vec{l}_3\big) =
    \begin{pmatrix}
      \sigma_{2}(\alpha)^{-1} & \sigma_{3}(\alpha)^{-1} & \alpha^{-1} \\
      1 & 1 & 1\\
      \alpha & \sigma_{2}(\alpha) & \sigma_{3}(\alpha)\\
    \end{pmatrix}
    =\Big(\sigma_{2}(\alpha)^{-1}\vec{l}_2,\,\sigma_{3}(\alpha)^{-1}\vec{l}_3,\,\alpha^{-1}\vec{l}_1\Big).
  \]
  Следовательно, $F_{3}$ --- палиндромическая симметрия  $\cf(l_1, l_2, l_3)$. Обратно, предположим, $F_{3}$ ---  палиндромическая симметрия  $\cf(l_1, l_2, l_3)$. Тогда существует $\mu_2$, такое что с точностью до перестановки индексов
  \[
    F_{3}\vec{l}_1 =
    \begin{pmatrix}
      \beta  \\
      1 \\
      \alpha
    \end{pmatrix}
    =\mu_2
    \begin{pmatrix}
      1  \\
      \sigma_{2}(\alpha) \\
      \sigma_{2}(\beta)
    \end{pmatrix},
  \]
  откуда $\mu_2 = \beta$, $\beta = \sigma_{2}(\alpha)^{-1}$ и $\norm(\alpha) = 1$. Стало быть, $\cf(l_1, l_2, l_3) \in \mathbf{CF}_{3}$.

  Пусть $\cf(l_1, l_2, l_3) \in \mathbf{CF}_{4}$. Тогда
  \[
    F_{4} \big(\vec{l}_1, \vec{l}_2, \vec{l}_3\big)=-
    \begin{pmatrix}
      \sigma_{2}(\alpha)^{-1} & \sigma_{3}(\alpha)^{-1} & \alpha^{-1} \\
      1 & 1 & 1\\
      \alpha & \sigma_{2}(\alpha) & \sigma_{3}(\alpha)\\
    \end{pmatrix}
    =-\Big(\sigma_{2}(\alpha)^{-1}\vec{l}_2,\,\sigma_{3}(\alpha)^{-1}\vec{l}_3,\,\alpha^{-1}\vec{l}_1\Big).
  \]
  Следовательно, $F_{4}$ --- палиндромическая симметрия  $\cf(l_1, l_2, l_3)$. Обратно, предположим, $F_{4}$ ---  палиндромическая симметрия  $\cf(l_1, l_2, l_3)$. Тогда существует $\mu_2$, такое что с точностью до перестановки индексов
  \[
    F_{4}\vec{l}_1 =
    \begin{pmatrix}
      \beta  \\
      -1 \\
      -\alpha
    \end{pmatrix}
    =\mu_2
    \begin{pmatrix}
      1  \\
      \sigma_{2}(\alpha) \\
      \sigma_{2}(\beta)
    \end{pmatrix},
  \]
  откуда $\mu_2 = \beta$, $\beta = -\sigma_{2}(\alpha)^{-1}$ и $\norm(\alpha) = -1$. Стало быть, $\cf(l_1, l_2, l_3) \in \mathbf{CF}_{4}$.
\end{proof}

\section{Доказательство критерия палиндромичности}\label{sec:proof}

В этом параграфе мы доказываем теорему \ref{main_t_1}. При помощи леммы \ref{l:CF_instead_of_statements} её можно переформулировать следующим образом: \emph{дробь $\cf(l_1,l_2,l_3)\in\gA_2$ палиндромична тогда и только тогда, когда она принадлежит одному из классов $\overline{\mathbf{CF}}_i$, $i=1,2,3,4$, причём $\overline{\mathbf{CF}}_2=\overline{\mathbf{CF}}_3=\overline{\mathbf{CF}}_4$.}

Если $\cf(l_1,l_2,l_3)$ принадлежит какому-то $\overline{\mathbf{CF}}_i$, то по лемме \ref{oper_eq} она палиндромична, ибо действие оператора из $\Gl_3(\Z)$ сохраняет палиндромичность.

Обратно, пусть дробь $\cf(l_1,l_2,l_3)\in\gA_2$ палиндромична и пусть $G$ --- её палиндромическая симметрия. Положим $F=G_+$ и рассмотрим точки $\vec z_1$, $\vec z_2$, $\vec z_3$, $\vec w$ из леммы \ref{main_lem}. Обозначим также через $\vec{e}_1$, $\vec{e}_2$, $\vec{e}_3$ стандартный базис $\R^3$. Для точек $\vec z_1$, $\vec z_2$, $\vec z_3$ выполняется либо утверждение \textup{(а)}, либо утверждение \textup{(б)} леммы \ref{main_lem}.

Пусть выполняется утверждение \textup{(а)} леммы \ref{main_lem}. Рассмотрим такой оператор $X_{1} \in \Gl_3(\Z)$, что
\[X_{1}\big(\vec{z}_1, \vec{z}_2, \vec{w}\big) = \big(\vec{e}_1 + \vec{e}_2, \vec{e}_1 - \vec{e}_3, \vec{e}_1\big).\]
Тогда $X_{1}(\vec{z}_3) = X_{1}(3\vec{w} - \vec{z}_1 - \vec{z}_2) = \vec{e}_1 - \vec{e}_2 + \vec{e}_3$ и $X_{1}FX_{1}^{-1} = F_{1}$, так как по лемме \ref{main_lem}
\[X_{1}FX_{1}^{-1}\big(\vec{e}_1 + \vec{e}_2, \vec{e}_1 - \vec{e}_3, \vec{e}_1 - \vec{e}_2 + \vec{e}_3\big) = \big(\vec{e}_1 - \vec{e}_3, \vec{e}_1 - \vec{e}_2 + \vec{e}_3, \vec{e}_1 + \vec{e}_2\big).\]
Стало быть, $X_{1}\big(\cf(l_1,l_2,l_3)\big) \in  \mathbf{CF}_{1}$, то есть $\cf(l_1,l_2,l_3)\in\overline{\mathbf{CF}}_1$.

Пусть выполняется утверждение \textup{(б)} леммы \ref{main_lem}. Рассмотрим такие операторы $X_2,X_3,X_4\in\Gl_3(\Z)$, что
\begin{align*}
  X_2\big(\vec{z}_1, \vec{z}_2, \vec{z}_3\big) & = \big(\vec{e}_1, \vec{e}_1 + \vec{e}_3, \vec{e}_1 + \vec{e}_2\big), \\
  X_3\big(\vec{z}_1, \vec{z}_2, \vec{z}_3\big) & = \big(\vec{e}_1, \vec{e}_2, \vec{e}_3\big), \\
  X_4\big(\vec{z}_1, \vec{z}_2, \vec{z}_3\big) & = \big(\vec{e}_1, -\vec{e}_2, \vec{e}_3\big).
\end{align*}
Тогда $X_iFX_i^{-1}=F_i$, $i=2,3,4$, так как по лемме \ref{main_lem}
\begin{align*}
  & X_{2}FX_{2}^{-1}\big(\vec{e}_1, \vec{e}_1 + \vec{e}_3, \vec{e}_1 + \vec{e}_2\big) = \big(\vec{e}_1 + \vec{e}_3, \vec{e}_1 + \vec{e}_2, \vec{e}_1\big), \\
  & X_{3}FX_{3}^{-1}\big(\vec{e}_1, \vec{e}_2, \vec{e}_3\big) = \big(\vec{e}_2, \vec{e}_3, \vec{e}_1\big), \\
  & X_{4}FX_{4}^{-1}\big(\vec{e}_1, -\vec{e}_2, \vec{e}_3\big) = \big(-\vec{e}_2, \vec{e}_3, \vec{e}_1\big).
\end{align*}
Стало быть, $X_i\big(\cf(l_1,l_2,l_3)\big) \in  \mathbf{CF}_i$, $i=2,3,4$, то есть $\cf(l_1,l_2,l_3)\in\overline{\mathbf{CF}}_i$, $i=2,3,4$. В частности, отсюда следует, что $\overline{\mathbf{CF}}_2=\overline{\mathbf{CF}}_3=\overline{\mathbf{CF}}_4$.

Теорема \ref{main_t_1} доказана.

\section{Сохранение конуса палиндромической симметрией в случае  $n = 2$ и $n = 3$}\label{sec_6}

Как нетрудно видеть, в предложении \ref{two_dimension} лишь два пункта эквивалентны, в то время как в теореме \ref{main_t_1} эквивалентны три пункта. То есть фактически при $n=2$ имеется три типа палиндромических симметрий, а при $n=3$ --- всего лишь два. Суть этого отличия в том, что пункт \textup{(г)} предложения \ref{two_dimension} реализует некоторую особенность одномерных цепных дробей, которой нет у двумерных.

Пусть $\cf(l_1,l_2)\in\gA_1$ и пусть подпространства $l_1$ и $l_2$ порождаются векторами $\vec l_1=(1,\alpha)$ и $\vec l_2=(1,\beta)$ соответственно. Пусть $G$ --- палиндромическая симметрия $\cf(l_1,l_2)$. Тогда $G(l_1)=l_2$, $G(l_2)=l_1$ (см. предложение \ref{prop:geometry_of_palindromicity}). Стало быть, существуют такие вещественные числа $\mu_1,\mu_2$, что матрица оператора $G$ в базисе $\vec l_1,\vec l_2$ имеет вид
\[
  \begin{pmatrix}
    0 & \mu_1 \\
    \mu_2 & 0
  \end{pmatrix}.
\]
Поскольку $\mu_1\mu_2=-\det G$, справедливо $G^2=-(\det G)I_2=\pm I_2$ (ср. с леммой \ref{ord_3}). Если $\det G=-1$, то есть если $\mu_1$ и $\mu_2$ одного знака, $G$ сохраняет два (противоположных) конуса из $\cC(l_1,l_2)$, отображая парус в каждом из них на себя, <<переворачивая>> его. При этом $-G$ сохраняет другие два противоположных конуса. Если же $\det G=1$, то $G$ переставляет конусы по кругу, не сохраняя ни один из них. Например, $G$ может просто осуществлять поворот плоскости на $90^\circ$ (см. рис. \ref{fig:intermediate_center}).

\begin{figure}[h]
  \centering
  \begin{tikzpicture}[scale=0.8]
    \draw[very thin,color=gray,scale=1] (-3.8,-1.8) grid (4.8,4.8);

    \draw[color=black] plot[domain=-5.1*25/144:2.3] (\x, {12*\x/5}) node[right]{$l_1$};
    \draw[color=black] plot[domain=-4.1:5.1] (\x, {-5*\x/12}) node[right]{$l_2$};

    \fill[blue!10!,path fading=north]
        (1+3.8,2+2.8) -- (1,0) -- (1,2) -- (1+2.8*2/5,2+2.8) -- cycle;
    \fill[blue!10!,path fading=east]
        (1+3.8,-3.8/2) -- (1,0) -- (1+3.8,2+2.8) -- cycle;
    \fill[blue!10!,path fading=north]
        (3.8/2,1+3.8) -- (0,1) -- (-2-1.8,1+3.8) -- cycle;
    \fill[blue!10!,path fading=west]
        (-2-1.8,1+3.8) -- (0,1) -- (-2,1) -- (-2-1.8,1+1.8*2/5) -- cycle;

    \draw[color=blue] (1+3.8,-3.8/2) -- (1,0) -- (1,2) -- (1+2.8*2/5,2+2.8);
    \draw[color=blue] (3.8/2,1+3.8) -- (0,1) -- (-2,1) -- (-2-1.8,1+1.8*2/5);

    \node[fill=blue,circle,inner sep=1.2pt] at (1,0) {};
    \node[fill=blue,circle,inner sep=1.2pt] at (1,2) {};
    \node[fill=blue,circle,inner sep=1.2pt] at (0,1) {};
    \node[fill=blue,circle,inner sep=1.2pt] at (-2,1) {};
    \node[fill=blue,circle,inner sep=0.8pt] at (1,1) {};
    \node[fill=blue,circle,inner sep=0.8pt] at (-1,1) {};
    \node[fill=blue,circle,inner sep=0.8pt] at (3,-1) {};
    \node[fill=blue,circle,inner sep=0.8pt] at (1,3) {};

    \draw (3.1,1.1) node[above right]{$\cK(C_1)$};
    \draw (-2+0.1,3.1) node[above right]{$\cK(C_2)$};

    \draw (6,1.5) node[right]
        {$\begin{pmatrix}
            0 & -1 \\
            1 & \phantom{-}0
          \end{pmatrix}
          \cK(C_1)=\cK(C_2)$};
  \end{tikzpicture}
  \caption{Палиндромическая симметрия, не сохраняющая никакой из конусов ($n=2$)} \label{fig:intermediate_center}
\end{figure}

Наличие именно такого рода симметрий описывается пунктом \textup{(г)} предложения \ref{two_dimension}. При этом симметрий другого типа, то есть с определителем $-1$, у дроби $\cf(l_1,l_2)$ может и не быть. Из результатов работы \cite{german_tlyustangelov_mjcnt} следует, что все палиндромические симметрии дроби $\cf(l_1,l_2)\in\gA_1$ имеют определитель $1$ (то есть никакая из них не сохраняет никакой конус из $\cC(l_1,l_2)$) тогда и только тогда, когда у периода цепной дроби числа $\alpha$ нельзя выбрать неполное частное, относительно которого период был бы симметричным (например, если период равен $(1,2,2,1)$).

В случае же $n=3$ для всякой палиндромичной двумерной цепной дроби $\cf(l_1,l_2,l_3)$ найдётся конус из $\cC(l_1,l_2,l_3)$, инвариантный относительно действия некоторой палиндромической симметрии. Более того, если $G$ --- произвольная палиндромическая симметрия дроби $\cf(l_1,l_2,l_3)\in\gA_2$, то ввиду следствия \ref{dif} (см. параграф \ref{sec_5}) существует конус $C \in \mathcal{C}(l_1, l_2, l_3)$, такой что $G_+(C)=C$ (см. рис. \ref{fig:3d}).

\begin{figure}[h]
  \centering
  \begin{tikzpicture}[x=10mm, y=7mm, z=-5mm]
    \begin{scope}[rotate around x=25]

    \draw [-stealth, thin, rotate=70] (2.8,-2.8, 0.3) arc [start angle=130, end angle=-165, x radius=0.2cm, y radius=0.4cm];
    \node[fill=white,circle,inner sep=1pt] at (2.346,2.346,-2.346) {};

    \fill[blue!05,opacity=0.6] (0,0,0) -- (3,0,0) -- (0,0,-3) -- cycle;
    \fill[blue!15,opacity=0.6] (0,0,0) -- (0,3,0) -- (0,0,-3) -- cycle;

    \fill[blue!15,opacity=0.6] (0,0,0) -- (0,-3,0) -- (-3,0,0) -- cycle;

    \draw [blue,thin] (-1.43,-1.43,1.43) --  (-1,-1,1);
    \draw [blue,thick] (-1.43,-1.43,1.43) -- (-2,-2,2);

    \fill[blue!25,opacity=0.7] (-7/12,-2/12,27/12) -- (-27/12,-7/12,2/12) -- (-2/12,-27/12,7/12) -- cycle;
    \draw [blue] (-7/12,-2/12,27/12) -- (-27/12,-7/12,2/12) -- (-2/12,-27/12,7/12) -- cycle;

    \node[fill=blue,circle,inner sep=1pt] at (-1,-1,1) {};

    \node[fill=blue,circle,inner sep=1pt] at (-7/12,-2/12,27/12) {};
    \node[fill=blue,circle,inner sep=1pt] at (-27/12,-7/12,2/12) {};
    \node[fill=blue,circle,inner sep=1pt] at (-2/12,-27/12,7/12) {};

    \draw [very thin] (-5,0,0) -- (5,0,0);
    \draw [very thin] (0,-5,0) -- (0,5,0);
    \draw [very thin] (0,0,0) -- (0,0,5);
    \draw [very thin] (0,0,-3) -- (0,0,-5);

    \draw [thin] (-3,0,0) -- (3,0,0);
    \draw [thin] (0,-3,0) -- (0,3,0);
    \draw [thin] (0,0,0) -- (0,0,3);
    \draw [very thin] (0,0,0) -- (0,0,-3);

    \draw [thin] (3,0,0) -- (0,0,-3);
    \draw [thin] (0,0,3) -- (0,-3,0);
    \draw [thin] (0,3,0) -- (3,0,0);
    \draw [thin] (0,3,0) -- (0,0,-3);
    \draw [thin] (0,0,3) -- (-3,0,0);
    \draw [very thin] (0,-3,0) -- (-3,0,0);

    \draw [blue,thin] (-1,-1,1) --  (1,1,-1);

    \fill[blue!30,opacity=0.7] (7/12,2/12,-27/12) -- (27/12,7/12,-2/12) -- (2/12,27/12,-7/12) -- cycle;
    \draw [blue] (7/12,2/12,-27/12) -- (27/12,7/12,-2/12) -- (2/12,27/12,-7/12) -- cycle;

    \node[fill=blue,circle,inner sep=1pt] at (1,1,-1) {};

    \node[fill=blue,circle,inner sep=1pt] at (7/12,2/12,-27/12) {};
    \node[fill=blue,circle,inner sep=1pt] at (27/12,7/12,-2/12) {};
    \node[fill=blue,circle,inner sep=1pt] at (2/12,27/12,-7/12) {};

    \draw (27/12-1/24,7/12,-2/12) node[right] {$\vec z_1$};
    \draw (2/12,27/12,-7/12) node[above left] {$\vec z_2$};
    \draw (7/12-2/24,2/12,-27/12) node[above right] {$\vec z_3$};
    \draw (1-1/24,1,-1-1/24) node[below right] {$\vec w$};

    \draw [blue,thick] (1,1,-1) -- (2.9,2.9,-2.9);

    \draw (2.9,2.9,-2.9) node[right] {$l$};     	
    \draw (5, 0, 0) node[right] {$l_1$};
    \draw (0, 5, 0) node[left] {$l_2$};
    \draw (0, 0,-5) node[right] {$l_3$};
    \draw (2.65,2,-2) node[right] {$G_+$};

    \fill[blue!15,opacity=0.7] (0,0,0) -- (0,3,0) -- (3,0,0) -- cycle;

    \fill[blue!07,opacity=0.7] (0,0,0) -- (-3,0,0) -- (0,0,3) -- cycle;
    \fill[blue!15,opacity=0.7] (0,0,0) -- (0,-3,0) -- (0,0,3) -- cycle;

    \end{scope}
  \end{tikzpicture}
  \caption{Палиндромическая симметрия $G_+$ и инвариантные конусы ($n=3$)} \label{fig:3d}{Точки $\vec z_1$, $\vec z_2$, $\vec z_3$, $\vec w$ --- из леммы \ref{main_lem}}
\end{figure}

При этом $G_-(C)=-C$, а оставшиеся шесть конусов из $\cC(l_1,l_2,l_3)$ переставляются оператором $G_-$ по циклу длины $6$. В частности, отсюда следует, что паруса в этих шести конусах имеют одинаковую ком\-би\-на\-тор\-но-це\-ло\-чис\-лен\-ную структуру. Если в то же время у дроби $\cf(l_1,l_2,l_3)$ найдётся палиндромическая симметрия, сохраняющая какой-нибудь конус из $\cC(l_1,l_2,l_3)$, отличный от $\pm C$, то паруса во всех восьми конусах будут иметь одинаковую комбинаторно-целочисленную структуру. Например, такое наблюдается в случае трёхмерного оператора Фибоначчи, подробно разобранном в работе \cite{korkina_2dim}.


\end{document}